\documentclass{amsart}[12pt]
 \newtheorem{theorem}{Theorem}[section]
 \newtheorem{corollary}[theorem]{Corollary}
 \newtheorem{lemma}[theorem]{Lemma}
 \newtheorem{proposition}[theorem]{Proposition}
 \theoremstyle{definition}
 
 \theoremstyle{remark}
 \newtheorem{remark}[theorem]{Remark}
 
 \numberwithin{equation}{section}

\begin{document}

%-------------------------------------------------------------------------
% editorial commands: to be inserted by the editorial office
%
%\firstpage{1} \volume{228} \Copyrightyear{2004} \DOI{003-0001}
%\seriesextraline{This is the Concrete Title of this Book\br H.E. R and S.T.C. W, Eds.}
%\firstpage{1}

\title[Volterra type operators]
 {Volterra type operators between Bloch type spaces and weighted Banach spaces}

%----------Author 1
\author[Qingze Lin]{Qingze Lin}

\address{School of Applied Mathematics,
Guangdong University of Technology,
Guangzhou, Guangdong, 510520,
P.~R.~China}

\email{gdlqz@e.gzhu.edu.cn}

\subjclass[2010]{Primary 47G10; Secondary 30H05}

\keywords{Volterra type operator, boundedness, compactness, weighted Banach space, Bloch type space}

%%% ----------------------------------------------------------------------

\begin{abstract}
When the weight $\mu$ is more general than normal, the complete characterizations in terms of the symbol $g$ and weights for the conditions of the boundedness and compactness of $T_g: H^{\infty}_\nu\rightarrow H^{\infty}_\mu$ and $S_g: H^{\infty}_\nu\rightarrow H^{\infty}_\mu$ are still unknown. Smith et al. firstly gave the sufficient and necessary conditions for the boundedness of Volterra type operators on Banach spaces of bounded analytic functions when the symbol functions are univalent. In this paper, continuing their lines of investigations, we give the complete characterizations of the conditions for the boundedness and compactness of Volterra type operators $T_g$ and $S_g$ between Bloch type spaces $\mathcal{B}^\infty_\nu$ and weighted Banach spaces $H^{\infty}_\nu$ with more general weights, which generalize their works.
\end{abstract}

%%% ----------------------------------------------------------------------
\maketitle
%%% ----------------------------------------------------------------------

\section{\bf Introduction}
Let $\mathbb{D}$ be the unit disk of the complex plane $\mathbb{C}$ and $H(\mathbb D)$ be the space consisting of all analytic functions on the unit disk. For any $g\in H(\mathbb D)$, the Volterra type operator $T_g$ and its companion operator $S_g$ are defined by
$$(T_gf)(z)=\int_0^z f(\omega)g'(\omega)d\omega\quad\text{and}\quad (S_gf)(z)=\int_0^z f'(\omega)g(\omega)d\omega,$$
respectively, where $z\in\mathbb D, f\in H(\mathbb D)\,.$

The Volterra type operator $T_g$ was firstly introduced by Pommerenke \cite{P} to study the exponentials of BMOA functions and in the meantime, proved that $T_g$ acting on Hardy-Hilbert space $H^2$ is bounded if and only if $g\in BMOA$\,. After his work, Aleman, Siskakis and Cima \cite{AC,AS} studied the boundedness and compactness of $T_g$ on Hardy space $H^p$, in which they showed that $T_g$ is bounded (compact) on $H^p,\ 0<p<\infty$, if and only if $g\in BMOA\ (g\in VMOA)$. For the related works, see \cite{LMN}. Furthermore, Aleman and Siskakis \cite{AS1} studied the boundedness and compactness of $T_g$ on Bergman spaces while Galanopoulos, Girela and Pel\'{a}ez \cite{GGP,GP} investigated the boundedness of $T_g$ and $S_g$ on Dirichlet type spaces and Xiao \cite{XJ} studied $T_g$ and $S_g$ on $Q_p$ spaces.

Recently, Lin et al. \cite{LIN} characterized the boundedness of $T_g$ and $S_g$ acting on the derivative Hardy spaces $S^p(1\leq p<\infty)$. Miihkinen \cite{SM} investigated the strict singularity of $T_g$ on Hardy space $H^p$ and showed that the strict singularity of $T_g$ coincides with its compactness on $H^p$. Mengestie \cite{TM} obtained a complete description of the boundedness and compactness of the product of Volterra type operators and composition operators on weighted Fock spaces. Furthermore, by applying the Carleson embedding theorem and the Littlewood-Paley formula, Constantin and Pel\'{a}ez \cite{CP} obtained the boundedness and compactness of $T_g$ on weighted Fock spaces and investigated the invariant subspaces of the classical Volterra operator $T_z$ on such spaces.

Let us recall that a {\it weight} $\nu$ is a non-negative continuous function defined on $\mathbb{D}$ such that $\nu(z)=\nu(|z|)$ and is decreasing.

The weighted \mbox{Banach} spaces $H^{\infty}_\nu$ and $H^{0}_\nu$ are defined, respectively, by
$$H^{\infty}_\nu=\{f\in H(\mathbb D):\ \|f\|_{H^{\infty}_\nu}:=\sup_{z\in \mathbb D}\{\nu(z)|f(z)|\}<\infty\}\,,$$
and
$$H^{0}_\nu=\{f\in H(\mathbb D):\ \lim_{|z|\rightarrow 1^-}\nu(z)|f(z)|=0\}\,.$$

The norms and essential norms of several bounded linear operators on such spaces had been extensively studied. For instance, Montes-Rodr\'{\i}guez \cite{AM} obtained the essential norms of weighted composition operators on $H^{\infty}_\nu$\,.

Closely related to the weighted \mbox{Banach} spaces $H^{\infty}_\nu$ are the Bloch type spaces $\mathcal{B}^{\infty}_\nu$ defined by
$$\mathcal{B}^{\infty}_\nu=\{f\in H(\mathbb D):\ \|f\|_{\mathcal{B}}:=|f(0)|+\sup_{z\in \mathbb D}\{\nu(z)|f'(z)|\}<\infty\}\,.$$
$\mathcal{B}:=\mathcal{B}^{\infty}_{(1-|z|^2)}$ is the classical Bloch space. Corresponding to the space $H^{0}_\nu$, the little Bloch type space $\mathcal{B}^{0}_\nu$ is defined by
$$\mathcal{B}^{0}_\nu=\{f\in H(\mathbb D):\ \lim_{|z|\rightarrow1^-}\nu(z)|f'(z)|=0\}\,.$$

A weight $\nu$ is said to be {\it typical} if $\lim_{|z|\rightarrow1^-}\nu(z)=0$\,. If $\nu$ is typical, then the bidual space of $H^{0}_\nu$ is $(H^{0}_\nu)^{**}=H^{\infty}_\nu$, what's more, the polynomials are dense in $H^{0}_\nu$ (see \cite{BBG}) and in the little Bloch type space $\mathcal{B}^{0}_\nu$\,.

A weight $\nu$ is called {\it analytic} if $\nu(z)=1/f(|z|)$ for some $f\in H(\mathbb{D})$ such that $|f(z)|\leq f(|z|)$ for all $z\in\mathbb{D}$\,.

Denote $\nu_{\alpha}(z)=(1-|z|^2)^\alpha\,,0\leq\alpha$. Following Shields and Williams \cite{SW,BCHMP}, we define that

(a) the weight $\nu$ satisfies {\it property (U)} if there exists a positive number $\alpha$ such that the function $r\mapsto \nu(r)/\nu_{\alpha}(r)$ is almost increasing, or equivalently, $\inf_{n}\frac{\nu(1-2^{-(n+1)})}{\nu(1-2^{-n})}>0$\,;

(b) the weight $\nu$ satisfies {\it property (L)} if there exists a positive number $\alpha$ such that the function $r\mapsto \nu(r)/\nu_{\alpha}(r)$ is almost decreasing, or equivalently,  there is a natural number $k$ such that $\limsup_{n}\frac{\nu(1-2^{-(n+k)})}{\nu(1-2^{-n})}<1$\,;

(c) the weight $\nu$ satisfies both properties (U) and (L) is called {\it normal}.

Obviously, the standard weights $\nu_{\alpha}=(1-|z|^2)^\alpha$ with $\alpha>0$ are normal and Hardy and Littlewood \cite{HL} proved that for all $\alpha>0$, $H^0_{\nu_{\alpha}}=\mathcal{B}^0_{\nu_{\alpha+1}}$ and by duality, $H^\infty_{\nu_{\alpha}}=\mathcal{B}^\infty_{\nu_{\alpha+1}}$\,. This result was extended by Lusky \cite{LUSKY} who proved that If the weight $\nu$ satisfies property (U), then $\nu$ satisfies property (L) if and only if $H^0_{\nu(z)}=\mathcal{B}^0_{\nu(z)(1-|z|^2)}$\,. In particular, if $\nu$ is normal, then $H^0_{\nu(z)}=\mathcal{B}^0_{\nu(z)(1-|z|^2)}$ and by duality, $H^\infty_{\nu(z)}=\mathcal{B}^\infty_{\nu(z)(1-|z|^2)}$\,.

Following Basallote et al. \cite{BCHMP}, we say that a weight $\nu$ is {\it quasi-normal} if $H^0_{\nu(z)}=\mathcal{B}^0_{\nu(z)(1-|z|^2)}$\,. Note that the quasi-normal weights $\nu$ are all typical and thus by duality, $H^\infty_{\nu(z)}=\mathcal{B}^\infty_{\nu(z)(1-|z|^2)}$ as well. To simplify the notation, we denote $\omega(z):=(1-|z|^2)\mu(z)$ throughout. Accordingly, Lusky's result \cite{LUSKY} can be restated: If the weight $\nu$ satisfies property (U), then $\nu$ satisfies property (L) if and only if $\nu$ is quasi-normal. It should be noticed that Basallote et al. \cite{BCHMP} characterized the boundedness, compactness and weak compactness of $T_g$ acting between different weighted Banach spaces of analytic functions with the quasi-normal weight, and what's more, they applied the characterization of compactness to analyze the behavior of semigroups of composition operators on such spaces.

To state the main results of this paper, we need the definitions of the so-called associated weight (see \cite{BBT}). For a given weight $\nu$, its associated weight $\widetilde{\nu}$ is defined by
$$\widetilde{\nu}:=\frac{1}{\sup_{z\in\mathbb{D}}\{|f(z)|:\ f\in H^\infty_\nu\,,\ \|f\|_{H^\infty_\nu}\leq1\}}\equiv\frac{1}{\|\delta_z\|_{H^\infty_\nu\rightarrow\mathbb{C}}}\,,$$
where $\delta_z:\ H^\infty_\nu\rightarrow\mathbb{C}$ is the evaluation linear functional. It is known that the following relations between $\nu$ and its associated weight $\widetilde{\nu}$ hold:

(1) $\widetilde{\nu}$ is also a weight and $\nu\leq\widetilde{\nu}$, what's more, $\|\cdot\|_{H^\infty_\nu}=\|\cdot\|_{H^\infty_{\widetilde{\nu}}}$\,;

(2) for any $a\in\mathbb{D}$, there exists an element $f_a$ in the closed unit ball of $H^\infty_\nu$ such that $|f_a(a)|=1/\widetilde{\nu}(a)$\,;

(3) if $\nu$ is typical, then so is $\widetilde{\nu}$, and in this case, it holds that
$$\widetilde{\nu}=\frac{1}{\sup_{z\in\mathbb{D}}\{|f(z)|:\ f\in H^0_\nu\,,\ \|f\|_{H^\infty_\nu}\leq1\}}\equiv\frac{1}{\|\delta_z\|_{H^0_\nu\rightarrow\mathbb{C}}}\,.$$

A weight $\nu$ is called {\it essential} if there is a constant $C\geq1$ such that
$$\nu(z)\leq\widetilde{\nu}\leq C\nu\quad \textrm{for all }z\in\mathbb{D}\,.$$
When a weight $\nu$ is essential, then in any formulas in this paper, $\widetilde{\nu}$ can be replaced by $\nu$ without changing the results.
It is well known that the weights $\nu$ satisfying property (U) are always essential. It is also known that if $\nu$ is analytic, then $\nu=\widetilde{\nu}$ and, in particular, $\nu$ is essential as well. For example, the analytic functions $f(z)=\left(1+\log\frac{1}{(1-z^2)}\right)^\alpha$, with $\alpha>0$, produce the essential weights $\nu_{\log,\alpha}=\left(1+\log\frac{1}{(1-|z|^2)}\right)^{-\alpha}$ which are analytic and in the meantime, satisfy property (U).

So far, when the weight $\mu$ is more general than normal, the complete characterizations in terms of the symbol $g$ and the weights for the conditions of the boundedness and compactness of $T_g: H^{\infty}_\nu\rightarrow H^{\infty}_\mu$ and $S_g: H^{\infty}_\nu\rightarrow H^{\infty}_\mu$ are still unknown. Although the necessary and sufficient condition for $T_g: H^{\infty}\rightarrow H^{\infty}$ to be bounded are characterized recently by \mbox{Contreras}, \mbox{Pel\'{a}ez}, \mbox{Pommerenke} and R\"{a}tty\"{a} \cite{CPPR}, the condition characterizing the symbol $g$ uses the Cauchy transforms which seems to be difficult to verified. However, Anderson, et al. \cite{AJS} conjectured that $T_g: H^{\infty}\rightarrow H^{\infty}$ is bounded if and only if
$$\sup_{0\leq\theta<2\pi}\int_0^1|g'(re^{i\theta})|dr<+\infty\,.$$

This condition is sufficient for the boundedness of $T_g: H^{\infty}\rightarrow H^{\infty}$, however, it was recently proved to be not necessary by Smith et al. in \cite{SSV}, where a counterexample was given. However, when the symbol $g$ is univalent, this conjecture was proved to be affirmative in that paper. Following their works, Eklund et. al. \cite{ELPSW} studied the boundedness and compactness of $T_g$ between $H_{\nu_{\alpha}}^\infty$ and $H^{\infty}$, where $g$ is univalent and $\nu_{\alpha}=(1-|z|^2)^\alpha$ with $0\leq\alpha<1$.

In this paper, using the ideas from \cite{ELPSW,SSV}, we firstly give the sufficient and necessary conditions for the boundedness and compactness of $T_g: H^{\infty}_\nu\rightarrow H^{\infty}_\mu$ when $\nu$ is analytic and $\log(g')\in \mathcal{B}$, and in  the meantime, characterize the boundedness and compactness of $S_g: H^{\infty}_\nu\rightarrow H^{\infty}_\mu$ when $\nu$ is analytic and $\log(g)\in \mathcal{B}$, where the weight $\mu$ is any weight and thus generalize their works. Then for some general weights, we characterize conditions for the boundedness and compactness of $T_g: H^{\infty}_\nu\rightarrow H^{\infty}_\mu$ and $S_g: H^{\infty}_\nu\rightarrow H^{\infty}_\mu$\,, and in this time, the key point in our study is the relationship between the growth of a function and the growth of its derivative, motivated by \cite{BCHMP}. Finally, the conditions for the boundedness and compactness of $S_g: H^{\infty}_\nu\rightarrow \mathcal{B}^{\infty}_\mu$, $S_g: \mathcal{B}^{\infty}_\nu\rightarrow H^{\infty}_\mu$ and $S_g: \mathcal{B}^{\infty}_\nu\rightarrow \mathcal{B}^{\infty}_\mu$ are also investigated.

\section{\bf The boundedness of Volterra type operators}
For any $0<\eta<\pi, 0<r\leq1$, and $0\leq \theta<2\pi$, denote by $\mathcal{B}(\Omega^{r}_{\eta,\theta})$ the set of all analytic functions $F$ defined in the open sector
$$\Omega^{r}_{\eta,\theta}:=\left\{z\in\mathbb{D}: 0<|z|<r,\ \theta-\frac{\eta}{2}<\arg{z}<\theta+\frac{\eta}{2}\right\}\,,$$
such that
$$|F'(z)|\leq\frac{C_F}{|z|},\quad\text{for } z\in \Omega^{r}_{\eta,\theta}\,.$$
Here, $C_F$ is a constant depending only on function $F$.

Using the similar arguments from \cite{SSV}, we have the following two lemmas.
\begin{lemma}[Theorem~1.2 of \cite{SSV}]\label{le1}
Let $0<\gamma<\eta<\pi$, $0\leq \theta<2\pi$ and $\varepsilon>0$. Then there exists a $\delta(\varepsilon)>0$ such that for any $F\in \mathcal{B}(\Omega^{1/2}_{\gamma,\theta})$, there is a real harmonic function $u$ defined in $\Omega^{1}_{\eta,\theta}$ such that:

(1) $\left|\textup{Re}(F(xe^{i\theta}))-u(xe^{i\theta})\right|\leq\varepsilon$,\quad for $x\in (0,\delta(\varepsilon)]$;

(2) $|\widetilde{u}(z)|\leq C(\varepsilon,\gamma,\eta,C_F)<\infty$,\quad for $z\in\Omega^{1}_{\eta,\theta}$, where $\widetilde{u}$ is the harmonic conjugate of $u$ such that $\widetilde{u}\left(\frac{1}{2}e^{i\theta}\right)=0$.
\end{lemma}

\begin{lemma}[Page~189 of \cite{SSV}]\label{le2}
Let $0<\gamma<\eta<\pi$, $0\leq \theta<2\pi$. If $\psi_{\eta,\theta}:\Omega^{1}_{\eta,\theta}\rightarrow\mathbb{D}$ is a conformal map such that $\psi_{\eta,\theta}\left(\frac{1}{2}e^{i\theta}\right)=0$ and $\psi_{\eta,\theta}(0)=e^{i\theta}$,
then there exists a constant $C_1(\gamma,\eta)$ such that
$$\frac{|\psi'_{\eta,\theta}(z)|}{1-|\psi_{\eta,\theta}(z)|^2}\leq\frac{C_1(\gamma,\eta)}{|z|}$$
for all $z\in \Omega^{1/2}_{\gamma,\theta}$\,.
\end{lemma}

Now we give the complete characterizations of the conditions for the boundedness of Volterra type operators $T_g: H^{\infty}_\nu\rightarrow H^{\infty}_\mu$\, when $\log(g')\in \mathcal{B}$\,.
\begin{theorem}\label{th1}
If $g\in H(\mathbb{D})$ such that $\log(g')\in \mathcal{B}$, the weight $\nu$ is analytic and $\mu$ is an arbitrary weight, then $T_g: H^{\infty}_\nu\rightarrow H^{\infty}_\mu$ is bounded if and only if
$$\sup_{0\leq t<1}\ \sup_{0\leq\theta<2\pi}\mu(t)\int_0^t\frac{|g'(re^{i\theta})|}{\nu(r)}dr<+\infty\,.$$
\end{theorem}
\begin{proof}
Assume first that there is a positive number $N$ such that
$$\sup_{0\leq t<1}\ \sup_{0\leq\theta<2\pi}\mu(t)\int_0^t\frac{|g'(re^{i\theta})|}{\nu(r)}dr\leq N\,.$$
For any $f\in H^{\infty}_\nu$, we have

\begin{equation}\begin{split}\nonumber
\|T_g(f)\|_{H^{\infty}_{\mu}}&=\sup_{z\in\mathbb{D}}\left\{\mu(z) \left|\int_0^zf(\zeta)g'(\zeta)d\zeta\right|\right\}\\
&=\sup_{0\leq R<1}\ \sup_{0\leq\theta<2\pi} \left\{\mu(R)\left|\int_0^{Re^{i\theta}}f(\zeta)g'(\zeta)d\zeta\right|\right\}\\
&=\sup_{0\leq R<1}\ \sup_{0\leq\theta<2\pi}\left\{\mu(R)\left|\int_0^{R}f(re^{i\theta})g'(re^{i\theta})e^{i\theta}dr\right|\right\}\\
&\leq \|f\|_{H^{\infty}_{\nu}}\cdot\sup_{0\leq R<1}\ \sup_{0\leq\theta<2\pi}\left\{\mu(R)\int_0^{R}\frac{|g'(re^{i\theta})|}{\nu(r)}dr\right\}\\
&\leq N\|f\|_{H^{\infty}_{\nu}}\,.
\end{split}\end{equation}
Accordingly, $T_g: H^{\infty}_\nu\rightarrow H^{\infty}_\mu$ is bounded.

Conversely, assume that
$$\sup_{0\leq t<1}\ \sup_{0\leq\theta<2\pi}\mu(t)\int_0^t\frac{|g'(re^{i\theta})|}{\nu(r)}dr=+\infty\,,$$
which is equivalent to
$$\limsup_{t\rightarrow1^-}\sup_{0\leq\theta<2\pi}\mu(t)\int_0^t\frac{|g'(re^{i\theta})|}{\nu(r)}dr=+\infty\,.$$
Then choose a nondecreasing positive sequence $\{t_{0,n}\}_{n=1}^\infty$, with $t_{0,n}\rightarrow1^-$, such that for each $t_{0,n}$, there exits an angle $\theta_n$, with $0\leq\theta_n<2\pi$, such that
$$\mu(t_{0,n})\int_0^{t_{0,n}}\frac{|g'(re^{i\theta_n})|}{\nu(r)}dr\geq n\,.$$

Define the function $F_n:\Omega^{1}_{\eta,\theta_n}\rightarrow\mathbb{D}$ by
$$F_n(z)=-i\log\left(g'\circ\psi_{\eta,\theta_n}(z)\right)\,.$$
Since $\log(g')\in\mathcal{B}$, there is a constant $K>0$ such that
$$|g''(z)|\leq \frac{K|g'(z)|}{(1-|z|^2)}\quad \text{for all } z\in \mathbb{D}\,.$$
By Lemma~\ref{le2}, we have
\begin{equation}\begin{split}\nonumber
|F'_n(z)|&=\frac{\left|g''(\psi_{\eta,\theta_n}(z))\right|}{\left|g'(\psi_{\eta,\theta_n}(z))\right|}\cdot\left|\psi'_{\eta,\theta_n}(z)\right|\leq\frac{K\left|\psi'_{\eta,\theta_n}(z)\right|}{1-\left|\psi_{\eta,\theta_n}(z)\right|^2}\leq\frac{K\cdot C_1(\gamma,\eta)}{|z|}\,,
\end{split}\end{equation}
for all $z\in \Omega^{1/2}_{\gamma,\theta_n}$\,.
Thus, the restriction of $F_n$ to $\Omega^{1/2}_{\gamma,\theta_n}$ belongs to $\mathcal{B}(\Omega^{1/2}_{\gamma,\theta_n})$\,.

By Lemma~\ref{le1} with $\varepsilon=\frac{\pi}{3}$, we have a positive number $\delta(\frac{\pi}{3})$ such that for any $F_n\in \mathcal{B}(\Omega^{1/2}_{\gamma,\theta_n})$, there exits a real harmonic function $u_n$ defined in $\Omega^{1}_{\eta,\theta_n}$ such that:

(1) $\left|\text{Re}(F_n(xe^{i\theta_n}))-u_n(xe^{i\theta_n})\right|\leq\frac{\pi}{3}$,\quad for $x\in (0,\delta(\frac{\pi}{3})]$;

(2) $|\widetilde{u}_n(z)|\leq C(\frac{\pi}{3},\gamma,\eta,K\cdot C_1(\gamma,\eta))=:C_2(\gamma,\eta)<\infty$,\quad for $z\in\Omega^{1}_{\eta,\theta_n}$, where $\widetilde{u}_n$ is the harmonic conjugate of $u_n$ such that $\widetilde{u}_n\left(\frac{1}{2}e^{i\theta_n}\right)=0$.

From (1) and the definition of conformal function $\psi_{\eta,\theta_n}$, we obtain that there is a real $r_\eta$ with $0<r_\eta<1$ such that

$$\left|\arg{g'(re^{i\theta_n})}-u_n(\psi^{-1}_{\eta,\theta_n}(re^{i\theta_n}))\right|\leq\frac{\pi}{3},\quad \text{for }r\in [r_\eta,1)\,.$$

Define
$$h_n(z):=e^{-i(u_n(z)+i\widetilde{u}_n(z))}\,,$$
and consider the sequence $\{f_n\}_{n=1}^\infty$ of functions given by
$$f_n(z):={h_n\circ\psi^{-1}_{\eta,\theta_n}(z)}\cdot {f_\nu(e^{-2\theta_ni}z)}\,,$$
where $f_\nu$ is an analytic function such that $v(z)=1/f_\nu(|z|)$ for all $z\in\mathbb{D}$\,.
Then by (2) and the fact that the weight $\nu$ is analytic, it follows that the sequence $\{f_n\}_{n=1}^\infty$ is bounded in $H^{\infty}_{\nu}$ with
\begin{equation}\begin{split}\nonumber
\|f_n\|_{H^{\infty}_{\nu}}&=\sup_{z\in\mathbb{D}}\left\{\nu(z)\left|{h_n\circ\psi^{-1}_{\eta,\theta_n}(z)}\cdot{f_\nu(e^{-2\theta_ni}z)}\right|\right\}\\
&=\sup_{z\in\mathbb{D}}\left\{\nu(z)\left|{h_n\circ\psi^{-1}_{\eta,\theta_n}(z)}\right|\cdot{f_\nu(|z|)}\right\}\\
&\leq\left\|h_n\circ\psi^{-1}_{\eta,\theta_n}(z)\right\|_{H^{\infty}}\leq e^{C_2(\gamma,\eta)}.
\end{split}\end{equation}

Now choose $n$ large enough so that $r_\eta\leq t_{0,n}<1$, we see that
\begin{equation}\begin{split}\nonumber
\|T_g(f_n)\|_{H^{\infty}_{\mu}}&=\sup_{z\in\mathbb{D}}\left\{\mu(z)\left|\int_0^zf_n(\zeta)g'(\zeta)d\zeta\right|\right\}\\
&\geq\mu(t_{0,n})\left|\int_0^{t_{0,n}}f_n(re^{i\theta_n})g'(re^{i\theta_n})e^{i\theta_n}dr\right|\\
&\geq\left|\text{Re}\left(\mu(t_{0,n})\int_{r_\eta}^{t_{0,n}}f_n(re^{i\theta_n})g'(re^{i\theta_n})dr\right)\right|\\
&\phantom{\geq\ }-\left|\text{Re}\left(\mu(t_{0,n})\int_0^{r_\eta}f_n(re^{i\theta_n})g'(re^{i\theta_n})dr\right)\right|
\end{split}\end{equation}

First, we observe that
\begin{equation}\begin{split}\nonumber
&\left|\text{Re}\left(\mu(t_{0,n})\int_0^{r_\eta}f_n(re^{i\theta_n})g'(re^{i\theta_n})dr\right)\right|\leq\mu(0)\int_0^{r_\eta}|f_n(re^{i\theta_n})||g'(re^{i\theta_n})|dr\\
&\leq\mu(0)\|f_n\|_{H^{\infty}_{\nu}}\int_0^{r_\eta}\frac{|g'(re^{i\theta_n})|}{\nu(r)}dr\leq\mu(0)\frac{e^{C_2(\gamma,\eta)}C_3(\eta)}{\nu(r_\eta)}\,,
\end{split}\end{equation}
where $C_3(\eta):=\sup_{|z|\leq r_\eta}\{|g'(z)|\}$\,.

Then, for the estimation of the other integral, we see that

\begin{equation}\begin{split}\nonumber
&\left|\text{Re}\left(\mu(t_{0,n})\int_{r_\eta}^{t_{0,n}}f_n(re^{i\theta_n})g'(re^{i\theta_n})dr\right)\right|\\
&=\left|\mu(t_{0,n})\text{Re}\left(\int_{r_\eta}^{t_{0,n}}\frac{g'(re^{i\theta_n})}{\nu(r)}e^{-i\left(u_n(\psi^{-1}_{\eta,\theta_n}(re^{i\theta_n}))+i\widetilde{u}_n(\psi^{-1}_{\eta,\theta_n}(re^{i\theta_n}))\right)}dr\right)\right|\\
&=\Bigg|\mu(t_{0,n})\int_{r_\eta}^{t_{0,n}}\frac{|g'(re^{i\theta_n})|}{\nu(r)}e^{\widetilde{u}_n\left(\psi^{-1}_{\eta,\theta_n}(re^{i\theta_n})\right)}\\
&\phantom{=\ }\cdot\text{Re}\left(e^{i\left(\arg(g'(re^{i\theta_n}))-u_n(\psi^{-1}_{\eta,\theta_n}(re^{i\theta_n}))\right)}\right)dr\Bigg|\\
&\geq\cos(\frac{\pi}{3})e^{-C_2(\gamma,\eta)}\mu(t_{0,n})\int_{r_\eta}^{t_{0,n}}\frac{|g'(re^{i\theta_n})|}{\nu(r)}dr\\
&\geq\cos(\frac{\pi}{3})e^{-C_2(\gamma,\eta)}\mu(t_{0,n})\left(\int_{0}^{t_{0,n}}\frac{|g'(re^{i\theta_n})|}{\nu(r)}dr-\int_{0}^{r_\eta}\frac{|g'(re^{i\theta_n})|}{\nu(r)}dr\right)\\
&\geq\cos(\frac{\pi}{3})e^{-C_2(\gamma,\eta)}\left(\mu(t_{0,n})\int_{0}^{t_{0,n}}\frac{|g'(re^{i\theta_n})|}{\nu(r)}dr-\mu(0)\int_{0}^{r_\eta}\frac{|g'(re^{i\theta_n})|}{\nu(r)}dr\right)\\
&\geq\cos(\frac{\pi}{3})e^{-C_2(\gamma,\eta)}\left(n-\mu(0)\int_{0}^{r_\eta}\frac{|g'(re^{i\theta_n})|}{\nu(r)}dr\right)\\
&\geq\cos(\frac{\pi}{3})e^{-C_2(\gamma,\eta)}\left(n-\mu(0)\frac{C_3(\eta)}{\nu(r_\eta)}\right)\,.
\end{split}\end{equation}
Accordingly, for $n$ large enough such that $r_\eta\leq t_{0,n}<1$, we obtain that
$$\|T_g(f_n)\|_{H^{\infty}_{\mu}}\geq\cos(\frac{\pi}{3})e^{-C_2(\gamma,\eta)}n-\left(e^{2C_2(\gamma,\eta)}+\cos(\frac{\pi}{3})\right)\frac{\mu(0)e^{-C_2(\gamma,\eta)}C_3(\eta)}{\nu(r_\eta)}\,.$$
Letting $n\rightarrow\infty$, it follows that $T_g: H^{\infty}_\nu\rightarrow H^{\infty}_\mu$ is not bounded.
\end{proof}

\begin{remark}\label{re1}
From the proof of Theorem~\ref{th1}, we see that the condition $$\sup_{0\leq t<1}\ \sup_{0\leq\theta<2\pi}\mu(t)\int_0^t\frac{|g'(re^{i\theta})|}{\nu(r)}dr<+\infty\,,$$
is sufficient for the boundedness of $T_g: H^{\infty}_\nu\rightarrow H^{\infty}_\mu$ without the requirement that $\log(g')\in \mathcal{B}$ and the weight $\nu$ is analytic .
\end{remark}

\begin{remark}\label{re2}
If $g\in H(\mathbb{D})$ is univalent, then by \cite{PB}, it holds that $\log(g')\in \mathcal{B}$, thus Theorem~\ref{th1} holds for the univalent case.
\end{remark}

\begin{remark}\label{re3}
It should be noticed that Basallote et. al. \cite{BCHMP} had given the following complete characterization of the boundedness of $T_g: H^{\infty}_\nu\rightarrow H^{\infty}_\mu$ for the general symbol $g\in H(\mathbb{D})$ when $\mu$ is quasi-normal, namely, the following statements are equivalent:

(1) $T_g: H^{\infty}_\nu\rightarrow H^{\infty}_\mu$ is bounded.

(2) $\sup_{z\in\mathbb{D}}(1-|z|^2)|g'(z)|{\mu(z)}/{\widetilde{\nu}(z)}<\infty\,.$

If, in addition, $\nu$ is a typical weight, then both (1) and (2) are equivalent to

(3) $T_g: H^{0}_\nu\rightarrow H^{0}_\mu$ is bounded.
\end{remark}

The following corollary is the most interesting part of Theorem~\ref{th1}.
\begin{corollary}\label{cor1}
If $g\in H(\mathbb{D})$ such that $\log(g')\in \mathcal{B}$ and the weight $\nu$ is analytic, then $T_g: H^{\infty}_\nu\rightarrow H^{\infty}$ is bounded if and only if
$$\sup_{0\leq\theta<2\pi}\int_0^1\frac{|g'(re^{i\theta})|}{\nu(r)}dr<+\infty\,.$$
\end{corollary}
\begin{proof}
By Theorem~\ref{th1},  $T_g: H^{\infty}_\nu\rightarrow H^{\infty}$ is bounded if and only if
$$\limsup_{t\rightarrow1^{-}}\sup_{0\leq\theta<2\pi}\int_0^t\frac{|g'(re^{i\theta})|}{\nu(r)}dr<+\infty\,.$$
Since $\int_0^t\frac{|g'(re^{i\theta})|}{\nu(r)}dr$ is increasing as $t\rightarrow1^{-}$, it follows that
\begin{equation}\begin{split}\nonumber
\limsup_{t\rightarrow1^{-}}\sup_{0\leq\theta<2\pi}\int_0^t\frac{|g'(re^{i\theta})|}{\nu(r)}dr&=\sup_{0\leq\theta<2\pi}\limsup_{t\rightarrow1^{-}}\int_0^t\frac{|g'(re^{i\theta})|}{\nu(r)}dr\\
&=\sup_{0\leq\theta<2\pi}\int_0^1\frac{|g'(re^{i\theta})|}{\nu(r)}dr.
\end{split}\end{equation}
\end{proof}

Now we consider the boundedness of the companion operators $S_g: H^{\infty}_\nu\rightarrow H^{\infty}_\mu$. In this case, we suppose that the weight $\nu$ satisfies property (U). Then by \cite[Lemma~5]{WOLF}, there exists a constant $C_\nu>0$ such that for any $f\in H^\infty_{\nu}$\,,
$$|f^{(n)}(z)|\leq\frac{C_\nu\|f\|_{H^\infty_{\nu}}}{\nu(z)(1-|z|^2)^n}\,,$$
for any $z\in\mathbb{D}$ and every non-negative integer $n$\,. Thus, we first give a sufficient condition for the boundedness of $S_g: H^{\infty}_\nu\rightarrow H^{\infty}_\mu$ for general $g\in H(\mathbb{D})$\,.

\begin{proposition}\label{pro1}
If the weight $\nu$ satisfies property (U) while $\mu$ is an arbitrary weight, then $S_g: H^{\infty}_\nu\rightarrow H^{\infty}_\mu$ is bounded if
$$\sup_{0\leq t<1}\ \sup_{0\leq\theta<2\pi}\mu(t)\int_0^t\frac{|g(re^{i\theta})|}{(1-r^2)\nu(r)}dr<+\infty\,.$$
\end{proposition}
\begin{proof}
Assume that the condition is satisfied, then there exists $N>0$ such that

$$\sup_{0\leq t<1}\ \sup_{0\leq\theta<2\pi}\mu(t)\int_0^t\frac{|g(re^{i\theta})|}{(1-r^2)\nu(r)}dr\leq N\,.$$
For any $f\in H^{\infty}_\nu$, we have

\begin{equation}\begin{split}\nonumber
\|T_g(f)\|_{H^{\infty}_{\mu}}&=\sup_{z\in\mathbb{D}}\left\{\mu(z) \left|\int_0^zf'(\zeta)g(\zeta)d\zeta\right|\right\}\\
&=\sup_{0\leq R<1}\ \sup_{0\leq\theta<2\pi}\left\{\mu(R)\left|\int_0^{Re^{i\theta}}f'(\zeta)g(\zeta)d\zeta\right|\right\}\\
&=\sup_{0\leq R<1}\ \sup_{0\leq\theta<2\pi}\left\{\mu(R)\left|\int_0^{R}f'(re^{i\theta})e^{i\theta}g(re^{i\theta})dr\right|\right\}\\
&\leq C_{\nu}\|f\|_{H^{\infty}_{\nu}}\cdot\sup_{0\leq R<1}\ \sup_{0\leq\theta<2\pi}\left\{\mu(R)\int_0^{R}\frac{|g(re^{i\theta})|}{(1-r^2)\nu(r)}dr\right\}\\
&\leq NC_{\nu}\|f\|_{H^{\infty}_{\nu}}\,,
\end{split}\end{equation}
which implies that $S_g: H^{\infty}_\nu\rightarrow H^{\infty}_\mu$ is bounded.
\end{proof}

Now, we give the complete characterization for the boundedness of $S_g: H^{\infty}_\nu\rightarrow H^{\infty}_\mu$ when $\log(g)\in \mathcal{B}$ and the weight $\nu$ is analytic and quasi-normal.

\begin{theorem}\label{th2}
If $g\in H(\mathbb{D})$ such that $\log(g)\in \mathcal{B}$, the weight $\nu$ is analytic and quasi-normal and $\mu$ is an arbitrary weight, then $S_g: H^{\infty}_\nu\rightarrow H^{\infty}_\mu$ is bounded if and only if
$$\sup_{0\leq t<1}\ \sup_{0\leq\theta<2\pi}\mu(t)\int_0^t\frac{|g(re^{i\theta})|}{(1-r^2)\nu(r)}dr<+\infty\,.$$
\end{theorem}
\begin{proof}
The proof of the sufficiency is similar to Proposition~\ref{pro1}\,.

Now, assume that $S_g: H^{\infty}_\nu\rightarrow H^{\infty}_\mu$ is bounded. choose the sequence $\{g_n\}_{n=1}^\infty$ of functions defined by
$$g_n(z):=\int_0^z\frac{f_n(\zeta)}{(1-e^{-2\theta_ni}\zeta^2)}d\zeta\,,$$
where $f_n$ is defined in the proof of Theorem~\ref{th1}, then we see that $\{g_n\}_{n=1}^\infty$ is a bounded sequence in space $\mathcal{B}^\infty_{(1-|z|^2)\nu(z)}$\,. Since the $\nu$ is quasi-normal, it follows that $\{g_n\}_{n=1}^\infty$ is a bounded sequence in space $H^\infty_{\nu}$\,. The remaining of the proof is similar to the proof in Theorem~\ref{th1} and hence, we omit it.
\end{proof}

The following corollary is the most interesting part of Theorem~\ref{th2}.

\begin{corollary}\label{cor2}
If $g\in H(\mathbb{D})$ such that $\log(g)\in \mathcal{B}$, the weight $\nu$ is analytic and quasi-normal, then $S_g: H^{\infty}_\nu\rightarrow H^{\infty}$ is bounded if and only if
$$\sup_{0\leq\theta<2\pi}\int_0^1\frac{|g(re^{i\theta})|}{(1-r^2)\nu(r)}dr<+\infty\,.$$
\end{corollary}

\begin{remark}\label{re4}
It should be noted that the conditions in Theorem~\ref{th2} and Corollary~\ref{cor2} are essentially new and different from those of Theorem~\ref{th1} and Corollary~\ref{cor1} which are satisfied by univalent functions.
\end{remark}

If $\mu$ is quasi-normal, then we can give the following complete characterizations of the boundedness of $S_g: H^{\infty}_\nu\rightarrow H^{\infty}_\mu$ for the general symbol $g\in H(\mathbb{D})$\,. The following lemma will be used in the proof Theorem~\ref{th3} below. The arguments are standard (see \cite{BCHMP} for example), but we provide the proof for the sake of completeness.
\begin{lemma}\label{le3}
Let the weights $\nu$ and $\mu$ be typical such that $S_g: H^{0}_\nu\rightarrow H^{0}_\mu$ is bounded. Then $S_g^{**}=S_g$ and therefore, $S_g: H^{\infty}_\nu\rightarrow H^{\infty}_\mu$ is bounded. The same statement for $S_g: \mathcal{B}^{0}_\nu\rightarrow \mathcal{B}^{0}_\mu$ is also true.
\end{lemma}
\begin{proof}
We just proved the lemma for $S_g: H^{0}_\nu\rightarrow H^{0}_\mu$, the other case follows similarly. If $S_g: H^{0}_\nu\rightarrow H^{0}_\mu$ is bounded, then $S_g^*: (H^{0}_\mu)^*\rightarrow (H^{0}_\nu)^*$ is bounded and so is $S_g^{**}: H^{\infty}_\nu\rightarrow H^{\infty}_\mu$\,. Now consider the evaluation linear functionals $\delta_z\in (H^{0}_\mu)^*$\,. It is known that such elements are dense in $(H^{0}_\mu)^*$ and for any $f\in H^{0}_\nu$, we have
$$<S_g^*(\delta_z),f>=<\delta_z,S_g(f)>=\int_0^zf'(\zeta)g(\zeta)d\zeta\,.$$
Now for $f\in H^{\infty}_\nu$, we have
$$<(S_g)^{**}(f),\delta_z>=<f,S_g^*(\delta_z)>\,.$$
Since the functions $f_r(z)=f(rz)$ converge to $f$ as $r\rightarrow1^-$ in the weak$^*$ topology, it follows that $<f_r,x^*>\rightarrow<f,x^*>$ for all $x^*\in (H^{0}_\nu)^*$\,. Thus,
$$<f,S_g^*(\delta_z)>=\lim_{r\rightarrow1^-}<f_r,S_g^*(\delta_z)>=\lim_{r\rightarrow1^-}<S_g (f_r),\delta_z>=\lim_{r\rightarrow1^-}\int_0^zf'(r\zeta)g(\zeta)d\zeta\,.$$
Since $f'_r\rightarrow f'$ locally uniformly in $\mathbb{D}$, it holds that
$$\lim_{r\rightarrow1^-}\int_0^zf'(r\zeta)g(\zeta)d\zeta=\int_0^zf'(\zeta)g(\zeta)d\zeta=<S_g(f),\delta_z>\,.$$
Accordingly, we obtain that $S_g^{**}=S_g$\,.
\end{proof}

\begin{theorem}\label{th3}
If $g\in H(\mathbb{D})$, $\nu$ satisfies property (U) and $\mu$ is quasi-normal, then the following conditions are equivalent:

(1) $S_g: H^{\infty}_\nu\rightarrow H^{\infty}_\mu$ is bounded.

(2) $\sup_{z\in\mathbb{D}}\frac{\mu(z)}{{\nu}(z)}|g(z)|<\infty\,.$

If, in addition, $\nu$ is a typical weight, then both (1) and (2) are equivalent to

(3) $S_g: H^{0}_\nu\rightarrow H^{0}_\mu$ is bounded.
\end{theorem}
\begin{proof}
Since $\mu$ is quasi-normal, $H^{\infty}_\mu=\mathcal{B}^{\infty}_{\omega}$\,. Thus the boundedness of $S_g: H^{\infty}_\nu\rightarrow H^{\infty}_\mu$ is equivalent to the boundedness of $M_gD: H^{\infty}_\nu\rightarrow H^{\infty}_{\omega}$, where $M_g$ is the multiplication operator and $D$ is the differentiation operator, respectively. Now by choosing $\varphi(z)\equiv z$ in \cite[Theorem~7]{MZ}, the boundedness of $M_gD: H^{\infty}_\nu\rightarrow H^{\infty}_{\omega}$ is equivalent to the following condition
$$\sup_{z\in\mathbb{D}}\frac{\omega(z)}{(1-|z|^2){\nu}(z)}|g(z)|<\infty\,,$$
which is equivalent to condition (2), that is, (1) and (2) are equivalent.

Assume that $\nu$ is typical, then since $\mu$ is quasi-normal, $\mu$ is also typical. Thus, by Lemma~\ref{le3}, (3) implies (1).

Now, it remains to show that (2) implies (3). For any $f\in H^0_{{\nu}}$\,, we want to show that $S_g(f)\in \mathcal{B}^0_{\omega}$, or equivalently, $M_gD(f)\in H^0_{\omega}$\,. For any polynomial $p\in H^0_{{\nu}}$, by (2), we see that
\begin{equation}\begin{split}\nonumber
\lim_{|z|\rightarrow1^-}\omega(z)|p'(z)||g(z)|&=\lim_{|z|\rightarrow1^-}(1-|z|^2)\mu(z)|p'(z)||g(z)|\\
&\leq \left(\sup_{z\in\mathbb{D}}\frac{\mu(z)}{{\nu}(z)}|g(z)|\right)\lim_{|z|\rightarrow1^-}{\nu}(z)(1-|z|^2)|p'(z)|=0\,.
\end{split}\end{equation}
Thus $M_gD(p)\in H^0_{\omega}$\,. Since the of all polynomials is dense in $H^0_{{\nu}}$, there exists a sequence of polynomials $\{p_n\}$ such that $\|p_n-f\|_{H^0_{{\nu}}}\rightarrow0$ as $n\rightarrow\infty$\,. Now by (1), $M_gD: H^{\infty}_{{\nu}}\rightarrow H^{\infty}_\omega$ is bounded, it follows that
$$\|M_gD(p_n)-M_gD(f)\|_{H^{\infty}_\omega}\leq \|M_gD\|\|p_n-f\|_{H^{\infty}_{{\nu}}}\rightarrow0\,,$$
as $n\rightarrow\infty$\,. Thus $M_gD(f)\in H^0_{\omega}$, since $H^0_{\omega}$ is a closed subspace of $H^\infty_{\omega}$\,.
\end{proof}

Now we are to investigate the conditions for the boundedness of Volterra type operators $T_g$ and $S_g$ between Bloch type spaces $\mathcal{B}^\infty_{\nu}$ and $\mathcal{B}^\infty_{\mu}$\,. Since the corresponding conditions for the boundedness and compactness of Volterra type operators $T_g$ between Bloch type spaces $\mathcal{B}^\infty_{\nu}$ and $\mathcal{B}^\infty_{\mu}$ had been studied in \cite{ELPSW}, thus, we concentrate on the companion operator $S_g$\,. First, we prove the following lemma.
\begin{lemma}\label{le4}
Let $\nu$ and $\mu$ be weights. Then the product of the multiplication operator $M_g$ and the differentiation operator $D$, namely, $M_gD$ is bounded from Bloch type space $\mathcal{B}^\infty_{\nu}$ into weighted Banach space $H^\infty_{\mu}$ if and only if
$$\sup_{z\in\mathbb{D}}\frac{\mu(z)}{\widetilde{\nu}(z)}|g(z)|<\infty\,.$$
\end{lemma}
\begin{proof}
Since the differentiation operator $D$ is an isometry from Bloch type space $\mathcal{B}^\infty_{\nu}$ onto weighted Banach space $H^\infty_{\nu}$\,, it follows that $M_gD: \mathcal{B}^\infty_{\nu}\rightarrow H^\infty_{\mu}$ is bounded if and only if $M_g: H^\infty_{\nu}\rightarrow H^\infty_{\mu}$ is bounded, which, according to \cite{CHD,ZN}, is equivalent to the condition:
$\sup_{z\in\mathbb{D}}|g(z)|{\mu(z)}/{\widetilde{\nu}(z)}<\infty\,.$
\end{proof}

Then by Lemma~\ref{le4}, using the similar proof of Theorem~\ref{th3}, we obtain
\begin{theorem}\label{th4}
If $g\in H(\mathbb{D})$, $\nu$ and $\mu$ are weights, then the following conditions are equivalent:

(1) $S_g: \mathcal{B}^{\infty}_\nu\rightarrow \mathcal{B}^{\infty}_\mu$ is bounded.

(2) $\sup_{z\in\mathbb{D}}|g(z)|{\mu(z)}/{\widetilde{\nu}(z)}<\infty\,.$

If, in addition, $\nu$ and $\mu$ are typical weights, then both (1) and (2) are equivalent to

(3) $S_g: \mathcal{B}^{0}_\nu\rightarrow \mathcal{B}^{0}_\mu$ is bounded.
\end{theorem}

\begin{theorem}\label{th5}
If $g\in H(\mathbb{D})$, $\nu$ satisfies property (U) and $\mu$ is a weight, then the following conditions are equivalent:

(1) $S_g: H^{\infty}_\nu\rightarrow \mathcal{B}^{\infty}_\mu$ is bounded.

(2) $\sup_{z\in\mathbb{D}}\frac{\mu(z)}{(1-|z|^2){\nu}(z)}|g(z)|<\infty\,.$

If, in addition, $\nu$ and $\mu$ are typical weights, then both (1) and (2) are equivalent to

(3) $S_g: H^{0}_\nu\rightarrow \mathcal{B}^{0}_\mu$ is bounded.
\end{theorem}

\begin{theorem}\label{th6}
If $g\in H(\mathbb{D})$, $\nu$ is a weight and $\mu$ is a quasi-normal weight, then the following conditions are equivalent:

(1) $S_g: \mathcal{B}^{\infty}_\nu\rightarrow H^{\infty}_\mu$ is bounded.

(2) $\sup_{z\in\mathbb{D}}\frac{(1-|z|^2)\mu(z)}{\widetilde{\nu}(z)}|g(z)|<\infty\,.$

If, in addition, $\nu$ is a typical weight, then both (1) and (2) are equivalent to

(3) $S_g: \mathcal{B}^{0}_\nu\rightarrow H^{0}_\mu$ is bounded.
\end{theorem}

However, when $\mu$ is not a quasi-normal weight, then by the similar proof in Theorem~\ref{th2}, we can give the following complete characterization for the boundedness of $S_g: \mathcal{B}^{\infty}_\nu\rightarrow H^{\infty}_\mu$ when $\nu$ is an analytic weight and $\log(g)\in \mathcal{B}$.

\begin{proposition}\label{pro2}
If $g\in H(\mathbb{D})$ such that $\log(g)\in \mathcal{B}$, the weight $\nu$ is analytic and $\mu$ is an arbitrary weight, then $S_g: \mathcal{B}^{\infty}_\nu\rightarrow H^{\infty}_\mu$ is bounded if and only if
$$\sup_{0\leq t<1}\ \sup_{0\leq\theta<2\pi}\mu(t)\int_0^t\frac{|g(re^{i\theta})|}{\nu(r)}dr<+\infty\,.$$
\end{proposition}

\section{\bf The compactness of Volterra type operators}
In this section, we firstly characterize the compactness of $T_g: H^{\infty}_\nu\rightarrow H^{\infty}_\mu$ and its companion operator $S_g: H^{\infty}_\nu\rightarrow H^{\infty}_\mu$\,. In the end, we characterize the corresponding questions for Bloch type spaces.

By standard arguments, we have the following lemma.
\begin{lemma}\label{le5}
Let $\nu$ and $\mu$ be weights. If  $T_g: H^{\infty}_\nu\rightarrow H^{\infty}_\mu$ is bounded, then the following two statements are equivalent:

(1) $T_g: H^{\infty}_\nu\rightarrow H^{\infty}_\mu$ is compact;

(2) For  any bounded sequence $\{f_n\}_{n=1}^\infty\subset H^{\infty}_\nu$ such that $f_n(z)\rightarrow0$ uniformly on any compact subset of $\mathbb{D}$ as $n\rightarrow\infty$, it holds that $\|T_g(f_n)\|_{H^{\infty}_\mu}\rightarrow0$ as $n\rightarrow\infty$.
\end{lemma}

\begin{theorem}\label{th7}
If $g\in H(\mathbb{D})$ such that $\log(g')\in \mathcal{B}$, the weight $\nu$ is analytic and $\mu$ is an arbitrary weight, then $T_g: H^{\infty}_\nu\rightarrow H^{\infty}_\mu$ is compact if and only if
$$\lim_{t_2\rightarrow1^-}\limsup_{t_1\rightarrow1^-}\sup_{0\leq\theta<2\pi}\mu(t_1)\int_{t_2}^{t_1}\frac{|g'(re^{i\theta})|}{\nu(r)}dr=0\,.$$
\end{theorem}
\begin{proof}
Assume first that
$T_g: H^{\infty}_\nu\rightarrow H^{\infty}_\mu$ is compact and so it is bounded.
If
$$c:=\lim_{t_2\rightarrow1^-}\limsup_{t_1\rightarrow1^-}\sup_{0\leq\theta<2\pi}\mu(t_1)\int_{t_2}^{t_1}\frac{|g'(re^{i\theta})|}{\nu(r)}dr>0\,,$$
then by Theorem~\ref{th1}, we observe that $c<+\infty$\,.

Let $\{t_{2,n}\}_{n=1}^\infty$ be a sequence so that $0<t_{2,n}<t_{2,n+1}<1$, $\lim_{n\rightarrow\infty}t_{2,n}=1$ and
\begin{equation}\begin{split}\nonumber
c&=\inf_{n\in\mathbb{N}}\limsup_{t_1\rightarrow1^-}\sup_{0\leq\theta<2\pi}\mu(t_1)\int_{t_{2,n}}^{t_1}\frac{|g'(re^{i\theta})|}{\nu(r)}dr\\
&=\lim_{n\rightarrow\infty}\limsup_{t_1\rightarrow1^-}\sup_{0\leq\theta<2\pi}\mu(t_1)\int_{t_{2,n}}^{t_1}\frac{|g'(re^{i\theta})|}{\nu(r)}dr\,.\\
\end{split}\end{equation}
Then for any $n\in \mathbb{N}$, we can choose $t_{1,n}$ with $t_{2,n}<t_{1,n}$ such that
\begin{equation}\begin{split}\nonumber
\sup_{0\leq\theta<2\pi}\mu(t_{1,n})\int_{t_{2,n}}^{t_{1,n}}\frac{|g'(re^{i\theta})|}{\nu(r)}dr>c-\frac{1}{n^2}\,.
\end{split}\end{equation}
Also, we can choose an angle $\theta_n$ with $0\leq \theta_n<2\pi$ such that
\begin{equation}\begin{split}\nonumber
(1-t^2_{1,n})^\beta\int_{t_{2,n}}^{t_{1,n}}\frac{|g'(re^{i\theta_n})|}{\nu(r)}dr>c-\frac{1}{n}\,.
\end{split}\end{equation}

Define the sequence $\{k_n\}_{n=1}^\infty$ by choosing $k_n$ as the largest positive integer such that $t_{2,n}^{k_n}\geq\frac{1}{3}$ for any $n\in \mathbb{N}$, then since $\lim_{n\rightarrow\infty}\frac{\log\frac{1}{3}}{\log t_{2,n}}=+\infty$, it follows that $\lim_{n\rightarrow\infty}k_n=+\infty$\,.

Now define the sequence $\{s_n(z)\}_{n=1}^\infty$ by
$$s_n(z):=f_n(z)z^{k_n},\quad z\in\mathbb{D}\,,$$
where $f_n(z)$ is defined in the proof of Theorem~\ref{th1}, that is,
$$f_n(z):={h_n\circ\psi^{-1}_{\eta,\theta_n}(z)}\cdot {f_\nu(e^{-2\theta_ni}z)}\,,$$
where $f_\nu$ is an analytic function such that $v(z)=1/f_\nu(|z|)$ for all $z\in\mathbb{D}$\,.
Thus,
$$\|s_n\|_{H^{\infty}_{\nu}}\leq e^{C_2(\gamma,\eta)},$$
and $s_n\rightarrow 0$ uniformly on any compact subset of $\mathbb{D}$\,.

Choosing $n$ large enough such that $t_{2,n}>r_\eta$, where $r_\eta$ is defined in the proof of Theorem~\ref{th1}, then
\begin{equation}\begin{split}\nonumber
\|T_g(s_n)\|_{H^{\infty}_{\mu}}&\geq\left|\mu(t_{1,n})\int_0^{t_{1,n}}f_n(re^{i\theta_n})g'(re^{i\theta_n})r^{k_n}dr\right|\\
&\geq\left|\text{Re}\left(\mu(t_{1,n})\int_0^{t_{1,n}}f_n(re^{i\theta_n})g'(re^{i\theta_n})r^{k_n}dr\right)\right|\\
&\geq\left|\text{Re}\left(\mu(t_{1,n})\int_{r_\eta}^{t_{1,n}}f_n(re^{i\theta_n})g'(re^{i\theta_n})r^{k_n}dr\right)\right|\\
&\phantom{\geq\ }-\left|\text{Re}\left(\mu(t_{1,n})\int_0^{r_\eta}f_n(re^{i\theta_n})g'(re^{i\theta_n})r^{k_n}dr\right)\right|
\end{split}\end{equation}

First, we have
\begin{equation}\begin{split}\nonumber
&\left|\text{Re}\left(\mu(t_{1,n})\int_0^{r_\eta}f_n(re^{i\theta_n})g'(re^{i\theta_n})r^{k_n}dr\right)\right|\\
&\leq\mu(0)\int_0^{r_\eta}|f_n(re^{i\theta_n})||g'(re^{i\theta_n})|r^{k_n}dr\\
&\leq\mu(0)\|f_n\|_{H^{\infty}_{\nu}}\int_0^{r_\eta}\frac{|g'(re^{i\theta_n})|}{\nu(r)}r^{k_n}dr\\
&\leq\frac{\mu(0)e^{C_2(\gamma,\eta)}C_3(\eta)r^{k_n}_{\eta}}{\nu(r_\eta)}\,,
\end{split}\end{equation}
where $C_3(\eta):=\sup_{|z|\leq r_\eta}\{|g'(z)|\}$\,.

then, to estimate the other integral, we have
\begin{equation}\begin{split}\nonumber
&\left|\text{Re}\left(\mu(t_{1,n})\int_{r_\eta}^{t_{1,n}}f_n(re^{i\theta_n})g'(re^{i\theta_n})r^{k_n}dr\right)\right|\\
&=\left|\mu(t_{1,n})\text{Re}\left(\int_{r_\eta}^{t_{1,n}}\frac{g'(re^{i\theta_n})}{\nu(r)}e^{-i\left(u_n(\psi^{-1}_{\eta,\theta_n}(re^{i\theta_n}))+i\widetilde{u}_n(\psi^{-1}_{\eta,\theta_n}(re^{i\theta_n}))\right)}r^{k_n}dr\right)\right|\\
&=\Bigg|\mu(t_{1,n})\int_{r_\eta}^{t_{1,n}}\frac{|g'(re^{i\theta_n})|}{\nu(r)}e^{\widetilde{u}_n\left(\psi^{-1}_{\eta,\theta_n}(re^{i\theta_n})\right)}\\
&\phantom{=\ }\cdot\text{Re}\left(e^{i\left(\arg(g'(re^{i\theta_n}))-u_n(\psi^{-1}_{\eta,\theta_n}(re^{i\theta_n}))\right)}\right)r^{k_n}dr\Bigg|\\
&\geq\cos(\frac{\pi}{3})e^{-C_2(\gamma,\eta)}\left|\mu(t_{1,n})\int_{r_\eta}^{t_{1,n}}\frac{|g'(re^{i\theta_n})|}{\nu(r)}r^{k_n}dr\right|\\
&\geq\cos(\frac{\pi}{3})r_{2,n}^{k_n}e^{-C_2(\gamma,\eta)}\left|\mu(t_{1,n})\int_{t_{2,n}}^{t_{1,n}}\frac{|g'(re^{i\theta_n})|}{\nu(r)}dr\right|\\
&\geq\frac{1}{3}\cos(\frac{\pi}{3})e^{-C_2(\gamma,\eta)}(c-\frac{1}{n})\,.
\end{split}\end{equation}
Therefore, it follows that
\begin{equation}\begin{split}\nonumber
&\limsup_{n\rightarrow\infty}\|T_g(s_n)\|_{H^{\infty}_{\mu}}\\
&\geq\lim_{n\rightarrow\infty}\left(\frac{1}{3}\cos(\frac{\pi}{3})e^{-C_2(\gamma,\eta)}(c-\frac{1}{n})-\frac{\mu(0)e^{C_2(\gamma,\eta)}C_3(\eta)r^{k_n}_{\eta}}{\nu(r_\eta)}\right)\\
&=\frac{1}{3}\cos(\frac{\pi}{3})e^{-C_2(\gamma,\eta)}c>0\,,
\end{split}\end{equation}
which, according to Lemma~\ref{le5}, is in contradiction with the assumption that $T_g: H^{\infty}_\nu\rightarrow H^{\infty}_\mu$ is compact.

Conversely, assume that
$$\lim_{t_2\rightarrow1^-}\limsup_{t_1\rightarrow1^-}\sup_{0\leq\theta<2\pi}\mu(t_1)\int_{t_2}^{t_1}\frac{|g'(re^{i\theta})|}{\nu(r)}dr=0\,.$$
To prove that $T_g: H^{\infty}_\nu\rightarrow H^{\infty}_\mu$ is compact, choose any sequence $\{f_n\}_{n=1}^\infty\subset H^{\infty}_{\nu}$ such that there exists a positive number $W$ such that $\sup_{n\in \mathbb{N}}\|f_n\|_{H^{\infty}_{\nu}}\leq W$ and $f_n\rightarrow 0$ locally uniformly in $\mathbb{D}$ as $n\rightarrow\infty$\,. Let $\epsilon>0$, then there exists $t_{2,\epsilon}$ with $0<t_{2,\epsilon}<1$ such that
$$\limsup_{t_1\rightarrow1^-}\sup_{0\leq\theta<2\pi}\mu(t_1)\int_{t_{2,\epsilon}}^{t_1}\frac{|g'(re^{i\theta})|}{\nu(r)}dr<\epsilon\,.$$
Moreover, there exists $t_{1,\epsilon}$ with $t_{2,\epsilon}<t_{1,\epsilon}<1$ such that
$$\sup_{0\leq\theta<2\pi}\mu(t_1)\int_{t_{2,\epsilon}}^{t_1}\frac{|g'(re^{i\theta})|}{\nu(r)}dr<2\epsilon,\quad \text{whenever } t_{1,\epsilon}<t_1<1\,.$$

Now choose $N_\epsilon>0$ such that
$$\sup_{|z|\leq t_{1,\epsilon}}|f_n(z)|<\frac{\epsilon}{M_{1,\epsilon}\cdot t_{1,\epsilon}},\quad \text{whenever } n>N_\epsilon\,,$$
where $M_{1,\epsilon}:=\sup_{0\leq|z|\leq t_{1,\epsilon}}\left\{|g'(z)|\right\}$\,.

Then for $n>N_\epsilon$, we have
\begin{equation}\begin{split}\nonumber
\|T_g(f_n)\|_{H^{\infty}_{\mu}}&=\sup_{0\leq\theta<2\pi}\sup_{0\leq R<1}\left\{\mu(R)\left|\int_0^{R}f_n(re^{i\theta})g'(re^{i\theta})e^{i\theta}dr\right|\right\}\\
&\leq\sup_{0\leq\theta<2\pi}\sup_{0\leq R<1}\left\{\mu(R)\int_0^{R}|f_n(re^{i\theta})||g'(re^{i\theta})|dr\right\}\\
&\leq\sup_{0\leq\theta<2\pi}\sup_{0\leq R\leq t_{1,\epsilon}}\left\{\mu(R)\int_0^{R}|f_n(re^{i\theta})||g'(re^{i\theta})|dr\right\}\\
&\phantom{\leq}+\sup_{0\leq\theta<2\pi}\sup_{t_{1,\epsilon}<R}\left\{\mu(R)\int_0^{t_{2,\epsilon}}|f_n(re^{i\theta})||g'(re^{i\theta})|dr\right\}\\
&\phantom{\leq}+\sup_{0\leq\theta<2\pi}\sup_{t_{1,\epsilon}<R}\left\{\mu(R)\int_{t_{2,\epsilon}}^R|f_n(re^{i\theta})||g'(re^{i\theta})|dr\right\}\\
&\leq \epsilon+\epsilon+2W\epsilon=2(1+W)\epsilon\,,
\end{split}\end{equation}
which implies that $\lim_{n\rightarrow\infty}\|T_g(f_n)\|_{H^{\infty}_{\mu}}=0$ and hence, by Lemma~\ref{le5}, $T_g: H^{\infty}_\nu\rightarrow H^{\infty}_\mu$ is compact.
\end{proof}

\begin{remark}\label{re5}
From the proof of Theorem~\ref{th7}, we also see that the condition $$\lim_{t_2\rightarrow1^-}\limsup_{t_1\rightarrow1^-}\sup_{0\leq\theta<2\pi}\mu(t_1)\int_{t_2}^{t_1}\frac{|g'(re^{i\theta})|}{\nu(r)}dr=0\,,$$
is sufficient for the compactness of $T_g: H^{\infty}_\alpha\rightarrow H^{\infty}_\beta$ without the requirement that $\log(g')\in \mathcal{B}$ and weight $\nu$ is analytic.
\end{remark}
\begin{remark}\label{re6}
It should be noticed that Basallote et. al. \cite{BCHMP} had given the following complete characterization of the compactness of $T_g: H^{\infty}_\nu\rightarrow H^{\infty}_\mu$ for the general symbol $g\in H(\mathbb{D})$ when $\mu$ is quasi-normal, namely, the following statements are equivalent:

(1) $T_g: H^{\infty}_\nu\rightarrow H^{\infty}_\mu$ is compact.

(2) $\lim_{|z|\rightarrow1^-}(1-|z|^2)|g'(z)|{\mu(z)}/{\widetilde{\nu}(z)}=0\,.$

If, in addition, $\nu$ is a typical weight, then both (1) and (2) are equivalent to

(3) $T_g: H^{0}_\nu\rightarrow H^{0}_\mu$ is compact.
\end{remark}

For the compactness of $S_g: H^{\infty}_\nu\rightarrow H^{\infty}_\mu$\,, by the similar proof in Theorem~\ref{th7} and the arguments in the proof of Proposition~\ref{pro1}, we first give the sufficient condition when the weight $\nu$ satisfies property (U):

\begin{proposition}\label{pro3}
If the weight $\nu$ satisfies property (U) while $\mu$ is an arbitrary weight, then $S_g: H^{\infty}_\nu\rightarrow H^{\infty}_\mu$ is compact if
$$\lim_{t_2\rightarrow1^-}\limsup_{t_1\rightarrow1^-}\sup_{0\leq\theta<2\pi}\mu(t_1)\int_{t_2}^{t_1}\frac{|g(re^{i\theta})|}{(1-r^2)\nu(r)}dr=0\,.$$
\end{proposition}

By considering the sequence $\{g_n\}_{n=1}^\infty$ of functions defined by
$$g_n(z):=\int_0^z\frac{s_n(\zeta)}{(1-e^{-2\theta_ni}\zeta^2)}d\zeta\,,$$
where $s_n$ is defined in the proof of Theorem~\ref{th7}, and by the arguments in the proof of Theorem~\ref{th2}, we have

\begin{theorem}\label{th8}
If $g\in H(\mathbb{D})$ such that $\log(g)\in \mathcal{B}$, the weight $\nu$ is analytic and quasi-normal and $\mu$ is an arbitrary weight, then $S_g: H^{\infty}_\nu\rightarrow H^{\infty}_\mu$ is compact if and only if
$$\lim_{t_2\rightarrow1^-}\limsup_{t_1\rightarrow1^-}\sup_{0\leq\theta<2\pi}\mu(t_1)\int_{t_2}^{t_1}\frac{|g(re^{i\theta})|}{(1-r^2)\nu(r)}dr=0\,.$$
\end{theorem}

Now for the compactness of the companion operators $S_g: H^{\infty}_\nu\rightarrow H^{\infty}_\mu$, we have the following characterizations for the general symbol $g\in H(\mathbb{D})$\,.
\begin{theorem}\label{th9}
If $g\in H(\mathbb{D})$ and $\nu$ is a weight, then

(1) $S_g: H^{\infty}_\nu\rightarrow H^{\infty}$ is compact if and only if $g=0$\,;

(2) If $\nu$ satisfies property (U) and $\mu$ is quasi-normal, then the following conditions are equivalent:

(2.1) $S_g: H^{\infty}_\nu\rightarrow H^{\infty}_\mu$ is compact.

(2.2) $\lim_{|z|\rightarrow1^-}\frac{\mu(z)}{{\nu}(z)}|g(z)|=0\,.$

If, in addition, $\nu$ is a typical weight, then both (2.1) and (2.2) are equivalent to

(2.3) $S_g: H^{0}_\nu\rightarrow H^{0}_\mu$ is compact.
\end{theorem}
\begin{proof}
(1) This follows from the fact that $H^\infty\subset H^{\infty}_\nu$ and that $S_g: H^{\infty}\rightarrow H^{\infty}$ is compact if and only if $g=0$ (see \cite{AJS}).

(2) Since $\mu$ is quasi-normal, $H^{\infty}_\mu=\mathcal{B}^{\infty}_{\omega}$\,. Thus the compactness of $S_g: H^{\infty}_\nu\rightarrow H^{\infty}_\mu$ is equivalent to the compactness of $M_gD: H^{\infty}_\nu\rightarrow H^{\infty}_{\omega}$\,. Now by choosing $\varphi(z)\equiv z$ in \cite[Theorem~8]{MZ}, the compactness of $M_gD: H^{\infty}_\nu\rightarrow H^{\infty}_{\omega}$ is equivalent to the following condition
$$\lim_{|z|\rightarrow1^-}\frac{\omega(z)}{(1-|z|^2){\nu}(z)}|g(z)|=0\,,$$
which is equivalent to condition (2.2), that is, (2.1) and (2.2) are equivalent.

Assume that $\nu$ is typical, then since $\mu$ is quasi-normal, $\mu$ is also typical. By Lemma~\ref{le3}, it follows easily that (2.3) implies (2.1).

Now, it remains to show that (2.1) implies (2.3). If $S_g: H^{\infty}_\nu\rightarrow H^{\infty}_\mu$ is compact, then $S_g: H^{\infty}_\nu\rightarrow H^{\infty}_\mu$ is bounded. Hence by Theorem~\ref{th3}, $S_g: H^{0}_\nu\rightarrow H^{0}_\mu$ is bounded. Since $H^{0}_\mu$ is a closed subspace of $H^\infty_\mu$, it holds that $S_g: H^{0}_\nu\rightarrow H^{0}_\mu$ is compact.
\end{proof}

Now we are to investigate the conditions for the compactness of Volterra type operators $T_g$ and $S_g$ between Bloch type spaces $\mathcal{B}^\infty_{\nu}$ and $\mathcal{B}^\infty_{\mu}$\,. Since the corresponding conditions for the compactness of Volterra type operators $T_g$ between Bloch type spaces $\mathcal{B}^\infty_{\nu}$ and $\mathcal{B}^\infty_{\mu}$ had been studied in \cite{ELPSW}, thus, we concentrate on the companion operator $S_g$\,. First, we prove the following lemma.
\begin{lemma}\label{le6}
Let $\nu$ and $\mu$ be weights. Then the product of the multiplication operator $M_g$ and the differentiation operator $D$, namely, $M_gD$ is compact from Bloch type space $\mathcal{B}^\infty_{\nu}$ into weighted Banach space $H^\infty_{\mu}$ if and only if
$$\lim_{|z|\rightarrow1^-}\frac{\mu(z)}{\widetilde{\nu}(z)}|g(z)|=0\,.$$
\end{lemma}
\begin{proof}
Since the differentiation operator $D$ is a bounded operator from Bloch type space $\mathcal{B}^\infty_{\nu}$ into weighted Banach space $H^\infty_{\nu}$ and the Volterra operator $T_z$ is a bounded operator from weighted Banach space $H^\infty_{\nu}$ into Bloch type space $\mathcal{B}^\infty_{\nu}$\,, it follows that $M_gD: \mathcal{B}^\infty_{\nu}\rightarrow H^\infty_{\mu}$ is compact if and only if $M_g: H^\infty_{\nu}\rightarrow H^\infty_{\mu}$ is compact, which, according to \cite{CHD}, is equivalent to the condition
$$\lim_{|z|\rightarrow1^-}\frac{\mu(z)}{\widetilde{\nu}(z)}|g(z)|=0\,.$$
\end{proof}

Then by Lemma~\ref{le6}, using the similar proof of Theorem~\ref{th9}, we obtain
\begin{theorem}\label{th10}
If $g\in H(\mathbb{D})$, $\nu$ and $\mu$ are weights, then the following conditions are equivalent:

(1) $S_g: \mathcal{B}^{\infty}_\nu\rightarrow \mathcal{B}^{\infty}_\mu$ is compact.

(2) $\lim_{|z|\rightarrow1^-}|g(z)|{\mu(z)}/{\widetilde{\nu}(z)}=0\,.$

If, in addition, $\nu$ and $\mu$ are typical weights, then both (1) and (2) are equivalent to

(3) $S_g: \mathcal{B}^{0}_\nu\rightarrow \mathcal{B}^{0}_\mu$ is compact.
\end{theorem}

\begin{theorem}\label{th11}
If $g\in H(\mathbb{D})$, $\nu$ satisfies property (U) and $\mu$ is a weight, then the following conditions are equivalent:

(1) $S_g: H^{\infty}_\nu\rightarrow \mathcal{B}^{\infty}_\mu$ is compact.

(2) $\lim_{|z|\rightarrow1^-}\frac{\mu(z)}{(1-|z|^2){\nu}(z)}|g(z)|=0\,.$

If, in addition, $\nu$ and $\mu$ are typical weights, then both (1) and (2) are equivalent to

(3) $S_g: H^{0}_\nu\rightarrow \mathcal{B}^{0}_\mu$ is compact.
\end{theorem}

\begin{theorem}\label{th12}
If $g\in H(\mathbb{D})$, $\nu$ is a weight and $\mu$ is a quasi-normal weight, then the following conditions are equivalent:

(1) $S_g: \mathcal{B}^{\infty}_\nu\rightarrow H^{\infty}_\mu$ is compact.

(2) $\lim_{|z|\rightarrow1^-}\frac{(1-|z|^2)\mu(z)}{\widetilde{\nu}(z)}|g(z)|=0\,.$

If, in addition, $\nu$ is a typical weight, then both (1) and (2) are equivalent to

(3) $S_g: \mathcal{B}^{0}_\nu\rightarrow H^{0}_\mu$ is compact.
\end{theorem}

However, when $\mu$ is not a quasi-normal weight, then by the similar proof in Theorem~\ref{th8}, we can give the following complete characterization for the compactness of $S_g: \mathcal{B}^{\infty}_\nu\rightarrow H^{\infty}_\mu$ when $\nu$ is an analytic weight and $\log(g)\in \mathcal{B}$.
\begin{proposition}\label{pro4}
If $g\in H(\mathbb{D})$ such that $\log(g)\in \mathcal{B}$, the weight $\nu$ is analytic and $\mu$ is an arbitrary weight, then $S_g: \mathcal{B}^{\infty}_\nu\rightarrow H^{\infty}_\mu$ is compact if and only if
$$\lim_{t_2\rightarrow1^-}\limsup_{t_1\rightarrow1^-}\sup_{0\leq\theta<2\pi}\mu(t_1)\int_{t_2}^{t_1}\frac{|g(re^{i\theta})|}{\nu(r)}dr=0\,.$$
\end{proposition}

\end{document}